\documentclass[11pt]{article}

\usepackage[activeacute,english]{babel}
\usepackage{latexsym}
\usepackage{amssymb,amsmath}
\usepackage{amsthm}
\usepackage{graphicx}

\usepackage{moreverb}
\usepackage[dvips,colorlinks,bookmarksopen,bookmarksnumbered,citecolor=red,urlcolor=red]{hyperref}

\theoremstyle{remark}
\newtheorem{remark}{Remark}
\newtheorem{example}{Example}
\newtheorem{theorem}{Theorem}

\newcommand{\bpsi}{\mbox{\boldmath$\psi$\unboldmath}}

\newcommand{\blambda}{\mbox{\boldmath$\lambda$\unboldmath}}
\newcommand{\bzeta}{\mbox{\boldmath$\zeta$\unboldmath}}
\newcommand{\bmu}{\mbox{\boldmath$\mu$\unboldmath}}
\newcommand{\bell}{\mbox{\boldmath$\ell$\unboldmath}}

\begin{document}

 \begin{center}
   {\sf ~\\[14pt]
 %%%%%% Please insert TITLE of your talk %%%%%%%
{\Large {\bf Boundary Element Procedure for 3D Electromagnetic Transmission Problems with Large Conductivity}}}
 \end{center}

 %%%%%% Names of authors and their addresses: %%%%%%
 %%%%%% (Obvious tweaks for a different number of authors.)

 \footnotesize{
   \begin{center}
     M. Maischak$^1$, Z. Nezhi$^2$, J. E. Ospino$^3$, E.~P.~Stephan$^2$ \\[14pt]

     $^1$Department of Mathematics Sciences, Brunel University, U.K. \\[3mm]
     $^2$Institute for Applied Mathematics, Leibniz University of Hannover, Hannover, Germany \\[3mm]
     $^3$Departamento de Matem\'{a}ticas y Estad\'{\i}stica, Fundaci\'{o}n Universidad del Norte, Barranquilla, Colombia.  \\[3mm]
     e-mail: jospino@uninorte.edu.co

   \end{center}
   }

 \normalsize
 \noindent
 %%%%%%%  Please insert ABSTRACT %%%%%%
\begin{center}
 {\Large{\bf Abstract:}}\\
\end{center}
We consider the scattering of time periodic electro-magnetic fields by metallic obstacles,
the eddy current problem. In this interface problem different sets of Maxwell equations
must be solved in the obstacle and outside, while the tangential components of both
electric and magnetic fields are continuous across the interface. We
describe an asymptotic procedure, which applies for large conductivity and
reflects the skin effect in metals. The key to our method is to introduce a special
integral equation procedure for the exterior boundary value problem
corresponding to perfect conductors. The asymptotic procedure
leads to a great reduction in complexity for the numerical solution since it involves solving
only the exterior boundary value problem. Furthermore we introduce a new fem/bem coupling procedure for the transmission problem and consider the implementation of the Galerkin elements for the perfect
conductor problem and present numerical experiments.\\

{\bf Key words}: Boundary element; asymptotic expansion; skin effect.\\
\vspace{0.2cm}

%MSC 2010 Mathematics Subject Classification: 65N06, 65L12.\\

\section{Introduction}
\label{sec:s0}
\vspace{-2pt}
We present asymptotic expansions with respect to inverse powers of conductivity
for the electrical and magnetical fields and report the algorithm of MacCamy
and Stephan \cite{MacCamyS} which allows to compute the expansion terms of the electrical field
in the exterior domain by solving sucessively only exterior problems (so-called perfect
conductor problems) with different data on the interface between conductor (metal) and
isolator (air). We solve these exterior problems numerically by applying the Galerkin
boundary element method to first kind boundary integral equations which were originally
introduced by MacCamy and Stephan in \cite{MacCamyP}. This system of integral equations
on the interface $\Sigma$ results from a single layer potential ansatz for the electrical field
and has unknown densities namely a vector field and a scalar function on $\Sigma$ which we
approximate with lower order Raviart Thomas elements and continous piecewise linear
functions on a regular, triangular mesh on $\Sigma$ As in the two dimensional case, investigated
by Hariharan \cite{Hariharan,Hariharan1} and MacCamy and Stephan \cite{MacCamyE}, the asymptotic procedure
gives for the computation of the solution of the transmission problem a great reduction
in complexity since it involves solving only the exterior problem and furthermore only
a few expansion terms must be computed. We describe in detail how to implement the boundary element method
for the perfect conductor problem. As an alternative to the asymptotic expansions for
the solution of the transmission problem we introduce a new finite element/boundary element
Galerkin coupling procedure which converges quasi-optimally in the energy norm.

\vspace{-6pt}

\section{Asymptotic expansion for large conductivity and skin effect}
\label{sec:s1}
\vspace{-2pt}

Let $\Omega_{-}$ be a bounded region in $\mathbb{R}^{3}$ representing a metallic conductor and
$\Omega_{+}:=\mathbb{R}^{3}\backslash  \overline{\Omega_{-}}$.
$\Omega_{+}$ representing air. The parameters $\varepsilon_{0}$,
$\mu_{0}$, $\sigma_{0}$ denote permittivity, permeability and conductivity. is assumed to
have zero conductivity in $\Omega_{+}$ with $\varepsilon$, $\mu$, $\sigma$ in $\Omega_{-}$. Let the incident
electric and magnetic fields, $\textbf{E}^{0}$ and $\textbf{H}^{0}$,
satisfy Maxwell's equations in air. The total fields $\textbf{E}$
and $\textbf{H}$ satisfy the same Maxwell's equations as
$\textbf{E}^{0}$ and $\textbf{H}^{0}$ in $\Omega_{+}$ but a
different set of equations in $\Omega_{-}$. Across the interface $\Sigma:=\partial
\Omega_{-}=\partial \Omega_{+}$, which is assumed to be a regular analytic surface, the tangential components of both
$\textbf{E}$ and $\textbf{H}$ are continuous. $\textbf{E}-\textbf{E}^{0}$ and $\textbf{H}-\textbf{H}^{0}$ represent the scattered fields. All fields are time-harmonic
with frequency $\omega$. As in \cite{MacCamyS} we neglect conduction (displacement) currents in air (metal). Then, with appropriate scaling, the eddy current problem is (see \cite{Stratton,Weggler}).\\

\textit{Problem $(\textbf{P}_{\alpha \beta})$}: Given $\alpha>0$ and
$\beta>0$, find $\textbf{E}$ and $\textbf{H}$ such that;
\begin{equation}\label{s1}
\begin{array}{lll}
    \mbox{curl}\hspace{0.1cm}\textbf{E}=\textbf{H}, & \mbox{curl}\hspace{0.1cm}\textbf{H}=\alpha^{2}\textbf{E} &  \mbox{in}\hspace{0.3cm} \Omega_{+}\hspace{0.3cm}\mbox{(air)}\\\\
    \mbox{curl}\hspace{0.1cm}\textbf{E}=\textbf{H}, & \mbox{curl}\hspace{0.1cm}\textbf{H}=i\beta^{2}\textbf{E} &  \mbox{in}\hspace{0.3cm} \Omega_{-}\hspace{0.3cm}\mbox{(metal)}\\\\
    \textbf{E}_{T}^{+}=\textbf{E}_{T}^{-}, & \textbf{H}_{T}^{+}=\textbf{H}_{T}^{-}, &
    \mbox{on}\hspace{0.3cm} \Sigma.
  \end{array}
\end{equation}
$$\dfrac{\partial}{\partial r}\textbf{E}(\textbf{x})-i\alpha\textbf{E}(\textbf{x})=O\left( \dfrac{1}{r^{2}}\right)\hspace{0.3cm}\mbox{with}\hspace{0.3cm}r=|\textbf{x}|,\hspace{0.3cm}\mbox{as}\hspace{0.3cm}|\textbf{x}|\rightarrow\infty. $$
Here $\alpha^{2}=\omega^{2}\mu_{0}\varepsilon_{0}$ and
$\beta^{2}=\omega\mu \sigma-i\omega^{2}\mu \varepsilon$ are
dimensionless parameters, and $\beta^{2}=\omega\mu \sigma>0$ if displacement currents are neglected in metal $(\varepsilon=0)$. The
subscript $T$ denotes tangential component and the superscripts plus
and minus denote limits from $\Omega_{+}$ and $\Omega_{-}$.\\
At higher frequencies the constant $\beta$ is usually large leading to the \textit{perfect conductor approximation}. Formally this
means solving only the $\Omega_{+}$ equation and requiring that
$\textbf{E}_{T}=0$ on $\Sigma$. If we let $\textbf{E}$ and
$\textbf{H}$ denote the scattered
fields, we obtain\\

\textit{Problem $(\textbf{P}_{\alpha \infty})$}: Given $\alpha>0$,
find $\textbf{E}$ and $\textbf{H}$ such that;
\begin{equation}\label{s2}
\begin{array}{rrr}
     \mbox{curl}\hspace{0.1cm}\textbf{E}=\textbf{H}, & \mbox{curl}\hspace{0.1cm}\textbf{H}=\alpha^{2}\textbf{E} &  \mbox{in}\hspace{0.3cm} \Omega_{+}\\\\
     & \textbf{E}_{T}=-\textbf{E}_{T}^{0}, & \mbox{on}\hspace{0.3cm} \Sigma.
  \end{array}
\end{equation}
\begin{remark}\label{rem1}
There exists at most one solution of problem $(\textbf{P}_{\alpha
\beta})$ for any $\alpha>0$ and $0<\beta\leq \infty$ (see \cite{Muller}).
\end{remark}
\begin{remark}
There exists a sequence $\{\alpha_{k}\}_{k=1}^{\infty}$, such that if $\alpha\neq\alpha_{k}$ then $\mbox{curl}\hspace{0.1cm}\textbf{E}=\textbf{H}$, $\mbox{curl}\hspace{0.1cm}\textbf{H}=\alpha^{2}\textbf{E}$ in $\Omega_{+}$, $\textbf{E}_{T}\equiv 0$ on $\Sigma$ implies $\textbf{E}\equiv\textbf{H}\equiv 0$ in $\Omega_{+}$.
\end{remark}

We are interesting in an asymptotic expansion of the solution of problem $(\textbf{P}_{\alpha \beta})$ with respect to inverse powers of conductivity. With $\tau$ denoting the
distance from $\Sigma$ measured into $\Omega_{-}$ along the normal
to $\Sigma$ the expansions reads:
\begin{equation}\label{s3}
\textbf{E}\sim
\textbf{E}^{0}+\sum_{n=0}^{\infty}\textbf{E}_{n}\beta^{-n}
\hspace{0.3cm} \mbox{in}\hspace{0.3cm} \Omega_{+}
\end{equation}
\begin{equation}\label{s4}
\textbf{H}\sim
\textbf{H}^{0}+\sum_{n=0}^{\infty}\textbf{H}_{n}\beta^{-n}
\hspace{0.3cm} \mbox{in}\hspace{0.3cm} \Omega_{+}
\end{equation}
\begin{equation}\label{s5}
\textbf{E}\sim e^{-\sqrt{-i}\beta
\tau}\sum_{n=0}^{\infty}\textbf{E}_{n}\beta^{-n} \hspace{0.3cm}
\mbox{in}\hspace{0.3cm} \Omega_{-}
\end{equation}
\begin{equation}\label{s6}
\textbf{H}\sim e^{-\sqrt{-i}\beta
\tau}\sum_{n=0}^{\infty}\textbf{H}_{n}\beta^{-n} \hspace{0.3cm}
\mbox{in}\hspace{0.3cm} \Omega_{-}
\end{equation}
Here $\textbf{E}_{n}$ and $\textbf{H}_{n}$ are
independent of $\beta$ which is proportional to $\sqrt{\sigma}$. The exponential in (\ref{s5}) and (\ref{s6})
represents \textit{the skin effect}. Next we present from \cite{MacCamyS} these expansions for the half-space case where the various coefficients can be computed recursively. Note
$\textbf{E}_{0}$ and $\textbf{H}_{0}$ in (\ref{s3}) and (\ref{s4})
is simply the perfect conductor approximation, that is, the solution
of $(\textbf{P}_{\alpha \infty})$. $\textbf{E}_{n}$ and $\textbf{H}_{n}$ in (\ref{s3}) and (\ref{s4}) can be calculated successively by solving a sequence of problems of the same form as
$(\textbf{P}_{\alpha \infty})$ but with boundary values determined
from earlier coefficients. The $\textbf{E}_{n}$ and $\textbf{H}_{n}$
in (\ref{s5}) and (\ref{s6}) are obtained by solving ordinary
differential equations in the variable $x_{3}$.\\
For the ease of the reader we present here for the half-space case $\Omega_{+}=\mathbb{R}^{3}_{+}$ i.e. $x_{3}>0$ and $\Omega_{-}=\mathbb{R}^{3}_{-}$ i.e. $x_{3}<0$ a formal procedure to compute $\textbf{E}_{n}$, $\textbf{H}_{n}$ which was given by MacCamy and Stephan \cite{MacCamyS}. They substitute in (\ref{s3})-(\ref{s6}) into $(\textbf{P}_{\alpha \beta})$ for $\Sigma=\mathbb{R}^{2}$ and equate coefficients of $\beta^{-n}$. Here we give a short description of their approach.\\
Let $\chi=e^{\sqrt{-i}\beta x_{3}}$ and decompose fields $\textbf{F}$ into tangential and normal components
\begin{equation}\label{s13}
\textbf{F}=\mathfrak{F}+f\textbf{e}_{3},\hspace{0.3cm}
\mathfrak{F}=\mathcal{F}^{1}\textbf{e}_{1}+\mathcal{F}^{2}\textbf{e}_{2},
\end{equation}
with orthogonal component $\mathfrak{F}^{\bot}=\textbf{e}_{3}\times \mathfrak{F}$, and unit vectors $\textbf{e}_{i}$ ($i=1,2,3$).\\
Then one computes with the surface gradient $grad_{T}$ for the rotation
\begin{equation}\label{s17}
\mbox{curl}\hspace{0.1cm}\textbf{F}=\mathfrak{F}^{\bot}_{x_{3}}-(\mbox{grad}_{T}\hspace{0.1cm}f)^{\bot}-(\mbox{div}\hspace{0.1cm}\mathfrak{F}^{\bot})\textbf{e}_{3}
\end{equation}
and
\begin{equation}\label{s18}
\mbox{curl}(\chi \textbf{F})=\chi[\sqrt{-i}\beta
\mathfrak{F}^{\bot}+\mathfrak{F}^{\bot}_{x_{3}}-(\mbox{grad}_{T}\hspace{0.1cm}f)^{\bot}-(\mbox{div}\hspace{0.1cm}\mathfrak{F}^{\bot})\textbf{e}_{3}].
\end{equation}
Now setting $\textbf{E}_{n}=\mathcal{E}_{n}+\ell_{n}\textbf{e}_{3}$ one obtains for $x_{3}<0$
\begin{equation}\label{s20}
\mbox{curl}\hspace{0.1cm}\textbf{E}\sim\chi\{\sqrt{-i}\beta
\mathcal{E}^{\bot}_{0}+\sum_{n=0}^{\infty}[\sqrt{-i}\mathcal{E}^{\bot}_{n+1}+\mathcal{E}^{\bot}_{n,x_{3}}-(\mbox{grad}_{T}\hspace{0.1cm}\ell_{n})^{\bot}-(\mbox{div}\hspace{0.1cm}\mathcal{E}^{\bot}_{n})\textbf{e}_{3}]\beta^{-n}\},
\end{equation}
and
\begin{equation}\label{s21}
\begin{aligned}
\mbox{curl}\hspace{0.1cm}\mbox{curl}\hspace{0.1cm}\textbf{E}&\sim\chi\left\lbrace i\beta^{2}\mathcal{E}_{0}-\sqrt{-i}\beta
\mathcal{E}_{0,x_{3}}+\sqrt{-i}\beta\mbox{div}\hspace{0.1cm}\mathcal{E}_{0}\textbf{e}_{3}+\sum_{n=0}^{\infty}\left[ i\beta\mathcal{E}_{n+1}-\sqrt{-i}\mathcal{E}_{n+1,x_{3}}\right. \right. \\\\
&-\sqrt{-i}\mbox{div}\hspace{0.1cm}\mathcal{E}_{n+1}\textbf{e}_{3}-\sqrt{-i}\beta\mathcal{E}_{n,x_{3}} -\mathcal{E}_{n,x_{3},x_{3}}+\mbox{div}\hspace{0.1cm}\mathcal{E}_{n,x_{3}}\textbf{e}_{3}+\sqrt{-i}\beta\mbox{grad}\hspace{0.1cm}\ell_{n}\\\\
&\left.\left.  +(\mbox{grad}_{T}\hspace{0.1cm}\ell_{n})_{x_{3}}+\mbox{div}\hspace{0.1cm}\mbox{grad}\hspace{0.1cm}\ell_{n}\textbf{e}_{3}\right] \beta^{-n}+\mbox{grad}\hspace{0.1cm}\mbox{div}\hspace{0.1cm}\beta^{-n}\textbf{e}_{3}\right\rbrace \\\\
&=\chi[i\beta^{2}\mathcal{E}_{0}+i\beta^{2}\ell_{0}\textbf{e}_{3}+i\beta\mathcal{E}_{1}+i\beta
\ell_{1}\textbf{e}_{3}+\sum_{n=0}^{\infty}(i\mathcal{E}_{n+2}+i\ell_{n+2}\textbf{e}_{3})\beta^{-n}]\sim i\beta^{2}\textbf{E}.
\end{aligned}
\end{equation}
Hence, equating coefficients of $\beta^{2}$ and $\beta$, respectively yields $\ell_{0}\equiv 0$, $i\ell_{1}=\sqrt{-i}\mbox{div}\hspace{0.1cm}\mathcal{E}_{0}$ and $\mathcal{E}_{0,x_{3}}=0$ implying $\mathcal{E}_{0}(x_{1},x_{2},x_{3})=\mathcal{E}_{0}(x_{1},x_{2},0)$.\\
As coefficients of $\beta^{0}$ one obtains
$$-\sqrt{-i}\mathcal{E}_{1,x_{3}}+\sqrt{-i}\mbox{grad}\hspace{0.1cm}\ell_{1}=0,$$
$$\sqrt{-i}\mbox{div}\hspace{0.1cm}\mathcal{E}_{1}+\mbox{div}\hspace{0.1cm}\mathcal{E}_{0,x_{3}}-\mbox{grad}\hspace{0.1cm}\mbox{div}\hspace{0.1cm}\mathcal{E}_{0}=i\ell_{2}.$$
Now the gauge condition $\mbox{div}\hspace{0.1cm}\mathcal{E}_{0}=0$  implies $\ell_{1}\equiv 0$ and $\mbox{div}\hspace{0.1cm}\mathcal{E}_{0,x_{3}}=0$, hence $\mathcal{E}_{1,x_{3}}=0$ and $\sqrt{-i}\mbox{div}\hspace{0.1cm}\mathcal{E}_{1}=i\ell_{2}.$\\
Thus $\mathcal{E}_{1}(x_{1},x_{2},x_{3})=\mathcal{E}_{1}(x_{1},x_{2},0)$.\\
Equating coefficients of $\beta^{-1}$ in (\ref{s21}) gives
$$-\sqrt{-i}\mathcal{E}_{2,x_{3}}-\sqrt{-i}\mathcal{E}_{2,x_{3}}+\sqrt{-i}\mbox{grad}\hspace{0.1cm}\ell_{2}=0,$$
$$\sqrt{-i}\mbox{div}\hspace{0.1cm}\mathcal{E}_{2}-\mbox{grad}\hspace{0.1cm}\mbox{div}\hspace{0.1cm}\mathcal{E}_{1}=i\ell_{3}.$$
Setting
\begin{equation}\label{s22}
\textbf{H}=\chi\sum_{n=0}^{\infty}(\mathcal{H}_{n}+h_{n}\textbf{e}_{3})\beta^{-n}
\end{equation}
MacCamy and Stephan obtain in \cite{MacCamyS} with $\ell_{1}=0$, $h_{0}=0$
$\mathcal{E}_{0}=0$:\\
\begin{equation}
\sqrt{-i}\mathcal{E}_{1}^{\bot}+\mathcal{E}_{0,x_{3}}^{\bot}=\mathcal{H}_{0},\hspace{0.2cm}\sqrt{-i}\mathcal{H}_{0}^{\bot}=i\mathcal{E}_{1},\hspace{0.2cm}h_{0}=\mbox{div}\hspace{0.1cm}\mathcal{E}_{0}^{\bot}=0.
\end{equation}
and
\begin{equation}\label{s27}
 \sqrt{-i}\mathcal{E}_{2}^{\bot}+\mathcal{E}_{1,x_{3}}^{\bot}=\mathcal{H}_{1},\hspace{0.3cm}
 \sqrt{-i}\mathcal{H}_{1}^{\bot}+\mathcal{H}_{0,x_{3}}^{\bot}=i\mathcal{E}_{2}
\end{equation}
\begin{equation}\label{s28}
 h_{1}=-\mbox{div}\hspace{0.1cm}\mathcal{E}_{1}^{\bot},\hspace{0.3cm}
 -\mbox{div}\hspace{0.1cm}\mathcal{H}_{0}^{\bot}=i\ell_{2}.
\end{equation}
and
\begin{equation}\label{s29}
\begin{array}{lcl}
   \mathcal{H}_{0,x_{3}}\equiv \mathcal{E}_{1,x_{3}}\equiv 0 \\\\
  \mathcal{H}_{0}\equiv \sqrt{-i}\mathcal{E}_{1}^{\bot} & \mbox{in}
  & x_{3}<0
\end{array}
\end{equation}
For $x_{3}>0$, we have with $\mbox{curl}\hspace{0.1cm}\textbf{E}=\textbf{H}$ yields
$$\mbox{curl}\hspace{0.1cm}\textbf{E}^{0}+\sum_{n=0}^{\infty}\mbox{curl}\hspace{0.1cm}\textbf{E}_{n}\beta^{-n}=\textbf{H}^{0}+\sum_{n=0}^{\infty}\textbf{H}_{n}\beta^{-n}$$
Equating coefficients of $\beta^{-n}$ one finds in $x_{3}>0$
$$\mbox{curl}\hspace{0.1cm}\textbf{E}^{0}=\textbf{H}^{0},\hspace{0.3cm}
 \mbox{curl}\hspace{0.1cm}\textbf{E}_{n}=\textbf{H}_{n},\hspace{0.2cm}n\geq
0,$$ (and correspondlying due to $\mbox{curl}\hspace{0.1cm}\textbf{H}=\alpha^{2}\textbf{E}$)
$$\mbox{curl}\hspace{0.1cm}\textbf{H}^{0}=\alpha^{2}\textbf{E}^{0},\hspace{0.3cm}
\mbox{curl}\hspace{0.1cm}\textbf{H}_{n}=\alpha^{2}\textbf{E}_{n},\hspace{0.2cm}n\geq
 0.$$
With the above relations the recursion process goes as follows. First one use (6.10) for
$n=0$ and (6.13), in \cite{MacCamyS}, to conclude that
$$
\begin{array}{cccc}
  \mbox{curl}\hspace{0.1cm}\textbf{E}_{0}=\textbf{H}_{0}, & \mbox{curl}\hspace{0.1cm}\textbf{H}_{0}=\alpha^{2}\textbf{E}_{0} & \mbox{in} & x_{3}>0 \\\\
  \textbf{E}_{0}^{+}=-(\textbf{E}^{0}_{T})^{-}, &  \mbox{on} & x_{3}=0. &
\end{array}
$$
Now $(\textbf{E}_{0},\textbf{H}_{0})$ is just the solution of
$(\textbf{P}_{\alpha\infty})$ which we can solve by the boundary integral equation procedure introduce in MacCamy and Stephan and revisited belov. But from
$(\ref{s1})_{3}$ we obtain
\begin{equation}\label{s30}
\mathcal{H}_{0}^{-}=\mathcal{H}_{0}^{+}=(\textbf{H}_{0})_{T}^{+}\hspace{0.3cm}\mbox{on}\hspace{0.3cm}x_{3}=0.
\end{equation}
Now the right side of (\ref{s30}) is known and easily computed. Then $(\ref{s1})_{3}$ and (\ref{s30}) yield
\begin{equation}\label{s31}
(\textbf{E}_{1})_{T}^{+}=(\textbf{E}_{1})_{T}^{-}=\mathcal{E}_{1}^{-}=-\sqrt{i}(\mathcal{H}_{0}^{\bot})^{-}=-\sqrt{i}((\textbf{H}_{0})_{T}^{+})^{\bot}.
\end{equation}
Therefore by (6.10), in \cite{MacCamyS}, we have a new again solvable problem for
$(\textbf{E}_{1},\textbf{H}_{1})$ which is just like
$(\textbf{P}_{\alpha\infty})$, that is
$$\mbox{curl}\hspace{0.1cm}\textbf{E}_{1}=\textbf{H}_{1},\hspace{0.3cm}\mbox{curl}\hspace{0.1cm}\textbf{H}_{1}=\alpha^{2}\textbf{E}_{1}\hspace{0.3cm}\mbox{in}\hspace{0.3cm}x_{3}>0,$$
but with new boundary values for $\textbf{E}_{T}$ as given by
(\ref{s31}).\\
For the complete algorithm see \cite{MacCamyS}. Note, with $\lambda=\sqrt{-i}$ we have $\mathcal{E}_{1}^{-}(x_{1},x_{2},0)=-\dfrac{1}{\lambda}(\textbf{n}\times\mbox{curl}\hspace{0.1cm}\textbf{E}_{0})$ yielding in $x_{3}<0$
$$\textbf{E}_{1}(x_{1},x_{2},x_{3})=\int_{0}^{-\tau}e^{\lambda\beta\widetilde{x}_{3}}\mathcal{E}_{1}^{-}(x_{1},x_{2},0)d\widetilde{x}_{3}=-\dfrac{1}{\lambda^{2}\beta}(\textbf{n}\times\mbox{curl}\hspace{0.1cm}\textbf{E}_{0})[e^{-\lambda\beta\tau}-1]$$

A comparison with Peron's results (see Chapter 5 in \cite{Peron}) shows that $\textbf{W}_{j}^{cd}(y_{\alpha},h_{\rho})=e^{-\sqrt{-i}\beta\tau}\textbf{E}_{j}$, $j\geq 0$, in $\Omega^{cd}$, $\lambda Y_{3}=\sqrt{-i}\beta\tau$ and $w_{j}=\ell_{j}$. Furthermore we see that the first terms in the asymptotic expansion of the electrical field for a smooth surface $\Sigma$ derived by Peron coincide with those for the half-space $x_{3}=0$ investigated by MacCamy and Stephan, namely $\ell_{0}=w_{0}=0$, $\ell_{1}=w_{1}=0$, $\mathcal{E}_{0}=\textbf{W}_{0}^{cd}=0$.
\begin{remark}
Since due Theorem 5 in Chapter 3 of \cite{Ospino} there exists only one solution of the electromagnetic transmission problem for a smooth interface this solution can be compute by the boundary integral equation procedure below, when we assume that (\ref{s47}) holds. Then for the electrical field $\textbf{E}$ obtained via the boundary integral equation system we have that in the tubular region $\Omega_{\pm}(\delta)=\left\lbrace x\in\Omega_{\pm},\mbox{dist}(x,\Sigma)<\delta\right\rbrace $ there holds for the remainders $\textbf{E}_{m}^{is(cd)}$ obtained by truncating (\ref{s3}) and (\ref{s5}) at $n=m$
$$\lVert\textbf{E}_{m,\rho}^{is}\rVert_{\textbf{W}(\mbox{curl},\Omega^{is})}\leq C_{1}\rho^{-m-1}\hspace{0.2cm}\mbox{and}\hspace{0.2cm}\lVert\textbf{E}_{m,\rho}^{cd}\rVert\leq C_{2}e^{C_{3}\tau}$$
for constants $C_{1},C_{2},C_{3}>0$, independent of $\rho$.
\end{remark}

\section{A boundary integral equation method of the first kind}
\label{sec:s2}

Next we describe the integral equation procedure for
$(\textbf{P}_{\alpha \beta})$ and $(\textbf{P}_{\alpha \infty})$ from \cite{MacCamyS,Weggler}.\\
Throughout the section we require that
\begin{equation}\label{s47}
\alpha\neq\alpha_{k},\hspace{0.3cm}k=1,2,\ldots
\end{equation}
This methods, like others, are based on the Stratton-Chu formulas from \cite{Stratton}. To describe these we need some notation. We will let $\textbf{n}$ denote the exterior normal to $\Sigma$. Given any vector field $\textbf{v}$ defined on $\Sigma$ we have
\begin{equation}\label{s48}
\textbf{v}=\textbf{v}_{T}+v_{N}\textbf{n},\hspace{0.3cm}\textbf{v}_{T}=\textbf{n}\times(\textbf{v}\times\textbf{n})
\end{equation}
where $\textbf{v}_{T}$, which lies in the tangent plane, is the tangential component of $\textbf{v}$.\\

We define the simple layer potential $\mathcal{V}_{\kappa}$ for density $\psi$ (correspondingly for a vector field) for the surface $\Sigma$ by
\begin{equation}\label{s50}
\mathcal{V}_{\kappa}(\psi)=\int_{\Sigma}\psi(\textbf{y})G_{\kappa}(|\textbf{x}-\textbf{y}|)ds_{y},G_{\kappa}(r)=\dfrac{e^{i\kappa r}}{4\pi r}.
\end{equation}
For a vector field $\textbf{v}$ on $\Sigma$ we define $\mathcal{V}_{\kappa}(\textbf{v})$ by (\ref{s50}) with $\textbf{v}$ replacing $\psi$.\\
We collect in the following lemma some of the well-known results about the simple layer potential $\mathcal{V}_{\kappa}$.
\begin{remark}\label{slem1}
\cite[Lemma 2.1]{MacCamyS} For any complex $\kappa$, $0\leq\mbox{arg}\kappa\leq\dfrac{\pi}{2}$ and any continuous $\psi$ on $\Sigma$; there holds:
\begin{itemize}
 \item[(i)] $\mathcal{V}_{\kappa}(\psi)$ is continuous in $\mathbb{R}^{3}$,
 \item[(ii)] $\Delta\mathcal{V}_{\kappa}(\psi)=-\kappa^{2}\mathcal{V}_{\kappa}(\psi)$ in $\Omega_{-}\cup\Omega_{+}$,
 \item[(iii)] $\mathcal{V}_{\kappa}(\psi)(\textbf{x})=O\left(\dfrac{e^{i\kappa |\textbf{x}|}}{|\textbf{x}|} \right) $ as $|\textbf{x}|\rightarrow\infty$,
 \item[(iv)]
$$\left( \dfrac{\partial\mathcal{V}_{\kappa}(\psi)}{\partial\textbf{n}}(\textbf{x})\right)^{\pm}=\mp\dfrac{1}{2}\psi(\textbf{x})+\int_{\Sigma}K_{\kappa}(\textbf{x},\textbf{y})\psi(\textbf{y})ds_{y},\hspace{0.3cm}\mbox{on}\hspace{0.3cm}\Sigma, $$
where $K_{\kappa}(\textbf{x},\textbf{y})=O(|\textbf{x}-\textbf{y}|^{-1})$ as $\textbf{y}\rightarrow\textbf{x}$.
 \item[(v)] $$(\textbf{n}\times\mbox{curl}\hspace{0.1cm}\mathcal{V}_{\kappa}(\textbf{v})(\textbf{x}))^{\pm}=\pm\dfrac{1}{2}\textbf{v}(\textbf{x})+\dfrac{1}{2}\int_{\Sigma}\textbf{K}_{\kappa}(\textbf{x},\textbf{y})\textbf{v}(\textbf{y})ds_{y},$$
where the matrix function $\textbf{K}_{\kappa}$ satisfies $\textbf{K}_{\kappa}(\textbf{x},\textbf{y})=O(|\textbf{x}-\textbf{y}|^{-1})$ as $\textbf{y}\rightarrow\textbf{x}$.
\end{itemize}
\end{remark}

For problem $(\ref{s1})_{2}$, in $\Omega_{-}$ the Stratton-Chu formula gives
\begin{equation}\label{s51}
\begin{array}{l}
\textbf{E}=\mathcal{V}_{\sqrt{i}\beta}(\textbf{n}\times\textbf{H})-\mbox{curl}\hspace{0.1cm}\mathcal{V}_{\sqrt{i}\beta}(\textbf{n}\times\textbf{E})+\mbox{grad}\hspace{0.1cm}\mathcal{V}_{\sqrt{i}\beta}(\textbf{n}\cdot\textbf{E}),\\\\
\textbf{H}=\mbox{curl}\hspace{0.1cm}\mathcal{V}_{\sqrt{i}\beta}(\textbf{n}\times\textbf{H})-\mbox{curl}\hspace{0.1cm}\mbox{curl}\hspace{0.1cm}\mathcal{V}_{\sqrt{i}\beta}(\textbf{n}\times\textbf{E}).
\end{array}
\end{equation}
Similarly, for problem $(\ref{s1})_{1}$, in $\Omega_{+}$
\begin{equation}\label{s52}
\begin{array}{l}
\textbf{E}=\mathcal{V}_{\alpha}(\textbf{n}\times\textbf{H})-\mbox{curl}\hspace{0.1cm}\mathcal{V}_{\alpha}(\textbf{n}\times\textbf{E})+\mbox{grad}\hspace{0.1cm}\mathcal{V}_{\alpha}(\textbf{n}\cdot\textbf{E}),\\\\
\textbf{H}=\mbox{curl}\hspace{0.1cm}\mathcal{V}_{\alpha}(\textbf{n}\times\textbf{H})-\mbox{curl}\hspace{0.1cm}\mbox{curl}\hspace{0.1cm}\mathcal{V}_{\alpha}(\textbf{n}\times\textbf{E}).
\end{array}
\end{equation}
For given $\textbf{n}\times\textbf{H}$, $\textbf{n}\times\textbf{E}$ and $\textbf{n}\cdot\textbf{E}$ (\ref{s52}) yield a solution of $(\textbf{P}_{\alpha \infty})$. But we know only $\textbf{n}\times\textbf{E}$. The standard treatment of $(\textbf{P}_{\alpha \infty})$ starts from (\ref{s52}), sets $\textbf{n}\times\textbf{H}=0$ and $\textbf{n}\cdot\textbf{E}=0$ and replaces $-\textbf{n}\times\textbf{E}$ by an unknown tangential field $\textbf{L}$ yielding
\begin{equation}\label{s53}
\textbf{E}=\mbox{curl}\hspace{0.1cm}\mathcal{V}_{\alpha}(\textbf{L}),\hspace{0.3cm}\textbf{H}=\mbox{curl}\hspace{0.1cm}\mbox{curl}\hspace{0.1cm}\mathcal{V}_{\alpha}(\textbf{L}).
\end{equation}
Then the boundary condition yields an integral equation of the second kind for $\textbf{L}$ in the tangent space to $\Sigma$.\\
The method (\ref{s53}) is analogous to solving the Dirichlet problem for the scalar Helmholtz equation with a double layer potential. But having found $\textbf{L}$ it is hard to determine $\textbf{H}_{T}$, or equivalently $\textbf{n}\times\textbf{H}$, on $\Sigma$. Note calculating $\textbf{n}\times\textbf{H}$ on $\Sigma$ involves finding a second normal derivative of $\mathcal{V}_{\alpha}(\textbf{L})$.\\
The method in \cite{MacCamyS} for $(\textbf{P}_{\alpha \infty})$ is analogous to solving the scalar problems with a simple layer potential (see \cite{Hsiao}). MacCamy and Stephan use (\ref{s52}) but this time they set $\textbf{n}\times\textbf{E}=0$ and replace $\textbf{n}\times\textbf{H}$ and $\textbf{n}\cdot\textbf{E}$ by unknowns $\textbf{J}$ and $M$. Thus they take
\begin{equation}\label{s54}
\textbf{E}=\mathcal{V}_{\alpha}(\textbf{J})+\mbox{grad}\hspace{0.1cm}\mathcal{V}_{\alpha}(M),\hspace{0.3cm}\textbf{H}=\mbox{curl}\hspace{0.1cm}\mathcal{V}_{\alpha}(\textbf{J}).
\end{equation}
If they can determine $\textbf{J}$ then in this case they can use Remark \ref{slem1} to determine $\textbf{n}\times\textbf{H}$, hence $\textbf{H}_{T}$ on $\Sigma$.\\
With the surface gradient $\mbox{grad}_{T}\psi=(\mbox{grad}\hspace{0.1cm}\psi)_{T}$ on $\Sigma$, the boundary condition in (\ref{s1}) and (\ref{s54}) imply, by continuity of $\mathcal{V}_{\alpha}$,
$$\textbf{n}\times\textbf{E}=\textbf{n}\times\mathcal{V}_{\alpha}(\textbf{J})+\textbf{n}\times\mbox{grad}\hspace{0.1cm}\mathcal{V}_{\alpha}(M)=-\textbf{n}\times\textbf{E}^{0}$$
or equivalently
\begin{equation}\label{s55}
\mathcal{V}_{\alpha}(\textbf{J})_{T}+\mbox{grad}_{T}\hspace{0.1cm}\mathcal{V}_{\alpha}(M)=-\textbf{E}^{0}_{T}.
\end{equation}
We note that for any field $\textbf{v}$ defined in a neighbourhood of $\Sigma$ one can define the surface divergence $\mbox{div}_{T}$ by
$$\mbox{div}\hspace{0.1cm}\textbf{v}=\mbox{div}_{T}\hspace{0.1cm}\textbf{v}+\dfrac{\partial v}{\partial\textbf{n}}\textbf{n}.$$
As shown in \cite[Lemma 2.3]{MacCamyS},
there holds for any differentiable tangential field $\textbf{v}$,
$\mbox{div}\hspace{0.1cm}\mathcal{V}_{\kappa}(\textbf{v})=\mathcal{V}_{\kappa}(\mbox{div}_{T}\hspace{0.1cm}\textbf{v})\hspace{0.3cm}\mbox{on}\hspace{0.3cm}\Sigma.$\\

Setting $\mbox{div}\textbf{E}=0$ on $\Sigma$ yields therefore with (\ref{s54})
$$0=\mbox{div}\hspace{0.1cm}\textbf{E}=\mbox{div}\hspace{0.1cm}\mathcal{V}_{\alpha}(\textbf{J})+\mbox{div}\hspace{0.1cm}\mbox{grad}\hspace{0.1cm}\mathcal{V}_{\alpha}(M)$$
and $\mbox{div}\hspace{0.1cm}\mbox{grad}\mathcal{V}_{\alpha}(M)=-\alpha^{2}\mathcal{V}_{\alpha}(M)$
gives immediately
\begin{equation}\label{s56}
\mathcal{V}_{\alpha}(\mbox{div}_{T}\hspace{0.1cm}\textbf{J})-\alpha^{2}\mathcal{V}_{\alpha}(M)=0.
\end{equation}

\section{FE/BE coupling}
\label{sec:s3}

Next we present a coupling method for the interface problem $(P_{\alpha\beta})$ (see \cite{Ammari,Ammari1,Hitmair, Hitmair1,Ospino}). Integration by parts gives in $\Omega_{-}$ for the first equation in $(P_{\alpha\beta})$ with $\gamma_{N}\textbf{E}=(\mbox{curl}\hspace{0.1cm}\textbf{E})\times\textbf{n}$, $\gamma_{D}\textbf{E}=\textbf{n}\times(\textbf{E}\times\textbf{n})$
\begin{equation}
\int_{\Omega_{-}}\mbox{curl}\hspace{0.1cm}\textbf{E}\cdot\mbox{curl}\hspace{0.1cm}\overline{\textbf{v}}d\textbf{x}-\int_{\Omega_{-}}i\beta^{2}\textbf{E}\cdot\overline{\textbf{v}}d\textbf{x}-\int_{\Sigma}\gamma_{N}^{-}\textbf{E}\cdot\gamma_{D}^{-}\overline{\textbf{v}}ds=0.
\end{equation}
Therefore with $\gamma_{N}^{-}\textbf{E}=\gamma_{N}^{+}\textbf{E}+\gamma_{N}\textbf{E}^{0}$ and setting $\textbf{E}=\mathcal{V}_{\alpha}(\textbf{J})+\mbox{grad}\hspace{0.1cm}\mathcal{V}_{\alpha}(M)$ in $\Omega_{+}$ we obtain
\small{$$\int_{\Omega_{-}}\mbox{curl}\hspace{0.1cm}\textbf{E}\cdot\mbox{curl}\hspace{0.1cm}\overline{\textbf{v}}d\textbf{x}-\int_{\Omega_{-}}i\beta^{2}\textbf{E}\cdot\overline{\textbf{v}}d\textbf{x}-\int_{\Sigma}\gamma_{N}^{+}(\mathcal{V}_{\alpha}(\textbf{J})+\mbox{grad}\hspace{0.1cm}\mathcal{V}_{\alpha}(M))\cdot\gamma_{D}^{+}\overline{\textbf{v}}ds=\int_{\Sigma}\gamma_{N}\textbf{E}^{0}\cdot\gamma_{D}^{+}\overline{\textbf{v}}ds.$$}
Note that $\gamma_{N}^{+}(\mathcal{V}_{\alpha}(\textbf{J})+\mbox{grad}\hspace{0.1cm}\mathcal{V}_{\alpha}(M))=\dfrac{1}{2}\textbf{J}+\dfrac{1}{2}\textbf{K}_{\alpha}(\textbf{J})$ where $\textbf{K}_{\alpha}$ is a smothing operator.\\
As shown in \cite[Lemma 4.5]{MacCamyS} there exists a continuous map $J_{\alpha}(\textbf{J})_{T}$ from $\textbf{H}^{r}(\Sigma)$ into $H^{r+1}(\Sigma)$, for any real number $r$ with
\begin{equation}\label{b4}
\mbox{div}_{T}\hspace{0.1cm}\mathcal{V}_{\alpha}(\textbf{J})_{T}=\mathcal{V}_{\alpha}(\mbox{div}_{T}\hspace{0.1cm}\textbf{J})+J_{\alpha}(\textbf{J})_{T}.
\end{equation}
As shown in \cite{MacCamyP} the system of boundary operators on $\Sigma$ (which is equivalent to (\ref{s55}) and (\ref{s56}))
\begin{equation}\label{b5}
\begin{array}{ll}
 \mathcal{V}_{\alpha}(\textbf{J})_{T}+\mbox{grad}_{T}\hspace{0.1cm}\mathcal{V}_{\alpha}(M)&=-\textbf{E}^{0}_{T}\\\\
 -J_{\alpha}(\textbf{J})_{T}-(\Delta_{T}+\alpha^{2})\mathcal{V}_{\alpha}(M)&=\mbox{div}_{T}\hspace{0.1cm}\textbf{E}^{0}_{T}.
 \end{array}
\end{equation}
is strongly elliptic as a mapping from $\textbf{H}^{-\frac{1}{2}}(\Sigma)\times H^{\frac{1}{2}}(\Sigma)$ into $\textbf{H}^{\frac{1}{2}}(\Sigma)\times H^{-\frac{1}{2}}(\Sigma)$, where $\mbox{grad}_{T}(\mbox{div}_{T})$ denote the surface gradient (surface divergence) and $\Delta_{T}$ the Laplace-Beltrami operator on $\Sigma$.\\
Now, our fem/bem coupling method is based on the variational formulation: For given incident field $\textbf{E}^{0}$ on $\Sigma$ find $\textbf{E}\in\textbf{H}(\mbox{curl},\Omega_{-})$, $\textbf{J}\in\textbf{H}^{-\frac{1}{2}}(\Sigma)$ and $M\in H^{\frac{1}{2}}(\Sigma)$ with
\small{
\begin{equation}\label{b6}
\begin{aligned}
\int_{\Omega_{-}}\mbox{curl}\hspace{0.1cm}\textbf{E}\cdot\mbox{curl}\hspace{0.1cm}\overline{\textbf{v}}d\textbf{x}-\int_{\Omega_{-}}i\beta^{2}\textbf{E}\cdot\overline{\textbf{v}}d\textbf{x}-\dfrac{1}{2}\int_{\Sigma}(\textbf{J}+\textbf{K}_{\alpha}(\textbf{J}))\cdot\gamma_{D}^{+}\overline{\textbf{v}}ds=\int_{\Sigma}\gamma_{N}\textbf{E}^{0}\cdot\gamma_{D}^{+}\overline{\textbf{v}}ds\\\\
\int_{\Sigma}\mathcal{V}_{\alpha}(\textbf{J})_{T}\cdot\overline{\textbf{j}}\hspace{0.1cm}dS+\int_{\Sigma}\mbox{grad}_{T}\mathcal{V}_{\alpha}(M)\cdot\overline{\textbf{j}}\hspace{0.1cm}dS=-\int_{\Sigma}\textbf{E}_{T}^{0}\cdot\overline{\textbf{j}}\hspace{0.1cm}dS,\\\\
-\int_{\Sigma}J_{\alpha}(\textbf{J})_{T}\overline{m}\hspace{0.1cm}dS-\int_{\Sigma}(\Delta_{T}+\alpha^{2})\mathcal{V}_{\alpha}(M)\overline{m}\hspace{0.1cm}dS=\int_{\Sigma}\mbox{div}_{T}\hspace{0.1cm}\textbf{E}^{0}_{T}\overline{m}dS,
\end{aligned}
\end{equation}
}
$\forall\textbf{v}\in\textbf{H}(\mbox{curl},\Omega_{-})$, $\textbf{j}\in\textbf{H}^{-\frac{1}{2}}(\Sigma)$, $m\in H^{\frac{1}{2}}(\Sigma)$.\\
In order to formulate a conforming Galerkin scheme for (\ref{b6}) we take subspaces $\textbf{H}^{1}_{h}\subset\textbf{H}(\mbox{curl},\Omega_{-})$, $\textbf{H}^{-\frac{1}{2}}_{h}\subset\textbf{H}^{-\frac{1}{2}}(\Sigma)$, $H^{\frac{1}{2}}_{h}\subset H^{\frac{1}{2}}(\Sigma)$ with mesh parameter $h$ and look for $\textbf{E}_{h}\in\textbf{H}^{1}_{h}$, $\textbf{J}_{h}\in\textbf{H}^{-\frac{1}{2}}_{h}$, $M_{h}\in H^{\frac{1}{2}}_{h}$ such that
\begin{equation}\label{b7}
\langle\mathcal{A}(\textbf{E}_{h},\textbf{J}_{h},M_{h}),(\textbf{v}_{h},\textbf{j}_{h},m_{h})\rangle=\langle\mathcal{F},(\textbf{v}_{h},\textbf{j}_{h},m_{h})\rangle
\end{equation}
where $\mathcal{A}$ is the operator given by the left hand side in (\ref{b6}), $\mathcal{F}=(\gamma_{N}\textbf{E}^{0},-\textbf{E}^{0}_{T},\mbox{div}_{T}\hspace{0.1cm}\textbf{E}^{0}_{T})$.
\begin{theorem}\label{ts1}
\begin{enumerate}
 \item System (\ref{b6}) has a unique solution $(\textbf{E},\textbf{J},M)$ in $\textbf{X}=\textbf{H}(\mbox{curl},\Omega_{-})\times\textbf{H}^{-\frac{1}{2}}(\Sigma)\times H^{\frac{1}{2}}(\Sigma)$.
 \item The Galerkin system (\ref{b7}) is uniquely solvable in $\textbf{X}_{h}=\textbf{H}^{1}_{h}\times\textbf{H}^{-\frac{1}{2}}_{h}\times H^{\frac{1}{2}}_{h}$ and there exists $C>0$, independent of $h$,
\begin{equation}\label{b8}
\begin{aligned}
\|\textbf{E}-\textbf{E}_{h}\|_{\textbf{H}(\mbox{curl},\Omega_{-})}+\|\textbf{J}-\textbf{J}_{h}\|_{\textbf{H}^{-\frac{1}{2}}(\Sigma)}+\|M-M_{h}\|_{H^{\frac{1}{2}}(\Sigma)}\\\\
\leq C\inf_{(\textbf{v},\textbf{j},m)\in\textbf{X}_{h}}\left\lbrace \|\textbf{E}-\textbf{v}\|_{\textbf{H}(\mbox{curl},\Omega_{-})}+\|\textbf{J}-\textbf{j}\|_{\textbf{H}^{-\frac{1}{2}}(\Sigma)}+\|M-m\|_{H^{\frac{1}{2}}(\Sigma)}\right\rbrace
\end{aligned}
\end{equation}
where $(\textbf{E},\textbf{J},M)$ and $(\textbf{E}_{h},\textbf{J}_{h},M_{h})$ solve (\ref{b6})-(\ref{b7}) respectively.
\end{enumerate}
\end{theorem}
\begin{proof}
First we note that system (\ref{b6}) is strongly elliptic in $\textbf{X}$ which follows by considering $\mathcal{A}$ as a system of pseudodifferential operators (cf. \cite{MacCamyP}). The only difference to \cite{MacCamyP} is that here we have additionally the first equation in (\ref{b6}). If we note $\Delta\textbf{E}=\mbox{curl}\mbox{curl}\textbf{E}-\mbox{grad}\mbox{div}\textbf{E}$ and take $\mbox{div}\textbf{E}=0$ we have that the principal symbol of $\mathcal{A}$ has the form (with $|\xi|^{2}=\xi_{1}^{2}+\xi_{2}^{2}$)
\begin{equation}\label{b9}
\sigma(\mathcal{A})(\xi)(\textbf{E},\textbf{J},M)^{t}=\left(
\begin{array}{cccccc}
|\xi|^{2}+\xi^{2}_{3} & 0 & 0 & 1 & 0 & 0\\\\
0 & |\xi|^{2}+\xi^{2}_{3} & 0 & 0 & 1 & 0\\\\
0 & 0 & |\xi|^{2}+\xi^{2}_{3} & 0 & 0 & 0\\\\
0 & 0 & 0 & \dfrac{1}{|\xi|} & 0 & i\xi_{1}\dfrac{1}{|\xi|}\\\\
0 & 0 & 0 & 0 & \dfrac{1}{|\xi|} & i\xi_{2}\dfrac{1}{|\xi|}\\\\
0 & 0 & 0 & 0 & 0 & |\xi|
\end{array}\right)\left( \begin{array}{c}
\\
E_{1}\\
\\
E_{2}\\
\\
E_{3}\\
\\
J^{1}\\
\\
J^{2}\\
\\
M\\
\\
\end{array}
\right)
\end{equation}
where $(E_{1},E_{2})=\textbf{E}_{T}$ and $E_{3}$ is perpendicular to $x_{3}=0$.\\
Obviously the two subblocks are strongly elliptic (see \cite{MacCamyP} for the lower subblock). Assuming that $(\alpha,\sqrt{i}\beta)$ is not an eigenvalue of $P_{\alpha\beta}$ we have existence and uniqueness of the exact solution. Due to the strong ellipticity of $\mathcal{A}$ there exists a unique Galerkin solution and the a priori error estimate holds due to the abstract results by Stephan and Wendland \cite{StephanW}.
\end{proof}

\section{Galerkin procedure for the perfect conductor problem ($P_{\alpha\infty}$)}
\label{sec:s4}

Next we consider the implementation of the Galerkin methods (see \cite{Christiansen,Ospino,Taskinen,Weggler}) and present corresponding numerical experiments for the integral equations (\ref{s55}) and (\ref{s56}). These experiments are performed with the program package \textit{Maiprogs}, cf. Maischak \cite{Mattias2,Mattias3}, which is a Fortran-based program package used for finite element and boundary element simulations \cite{Mattias4}. Initially developed by M. Maischak, \textit{Maiprogs} has been extended for electromagnetics problem by Teltscher \cite{Teltscher} and Leydecker \cite{Leydecker}.\\
We will investigate the exterior problem $(P_{\alpha\infty})$ by performing the integral equations procedure (\ref{s55}) and (\ref{s56}):\\

Testing against arbitrary functions $\textbf{j}\in
\textbf{H}^{-\frac{1}{2}}(\Sigma)$ and $m\in H^{\frac{1}{2}}(\Sigma)$ in
(\ref{s55}) and (\ref{s56}), we get
\begin{equation}\label{eg2}
\begin{aligned}
\int_{\Sigma}\mathcal{V}_{\alpha}(\textbf{J})_{T}\cdot\overline{\textbf{j}}\hspace{0.1cm}dS+\int_{\Sigma}\mbox{grad}_{T}\mathcal{V}_{\alpha}(M)\cdot\overline{\textbf{j}}\hspace{0.1cm}dS&=-\int_{\Sigma}\textbf{E}_{T}^{0}\cdot\overline{\textbf{j}}\hspace{0.1cm}dS,\\\\
-\int_{\Sigma}\mathcal{V}_{\alpha}(\mbox{div}_{T}\textbf{J})\cdot
\overline{m}\hspace{0.1cm}dS+\alpha^{2}\int_{\Sigma}\mathcal{V}_{\alpha}(M)\cdot
\overline{m}\hspace{0.1cm}dS&=0.
\end{aligned}
\end{equation}
Partial integration in the second term of $(\ref{eg2})_{1}$
$$\int_{\Sigma}\mbox{grad}_{T}\mathcal{V}_{\alpha}(M)\cdot\overline{\textbf{j}}\hspace{0.1cm}dS=-\int_{\Sigma}\mathcal{V}_{\alpha}(M)\cdot\mbox{div}_{T}\overline{\textbf{j}}\hspace{0.1cm}dS$$
shows that the formulation (\ref{eg2}) is symmetric: By definition
of symmetric bilinear forms $a$, $c$, of the bilinear form $b$ and
linear form $\ell$ through
\begin{equation*}
\begin{aligned}
a(\textbf{J},\textbf{j}):&=\int_{\Sigma}\mathcal{V}_{\alpha}(\textbf{J})_{T}\cdot\overline{\textbf{j}}\hspace{0.1cm}dS,\\\\
b(\textbf{J},m):&=-\int_{\Sigma}\mathcal{V}_{\alpha}(\mbox{div}_{T}\textbf{J})\cdot \overline{m}\hspace{0.1cm}dS\\\\
&=-\int_{\Sigma}\mathcal{V}_{\alpha}(m)\cdot\mbox{div}_{T}\overline{\textbf{J}}\hspace{0.1cm}dS,\\\\
c(M,m):&=\alpha^{2}\int_{\Sigma}\mathcal{V}_{\alpha}(M)\cdot \overline{m}\hspace{0.1cm}dS,\\\\
\ell(\textbf{j}):&=-\int_{\Sigma}\textbf{E}_{T}^{0}\cdot\overline{\textbf{j}}\hspace{0.1cm}dS
\end{aligned}
\end{equation*}
the variational formulation has the form: Find $(\textbf{J},M)\in
\textbf{H}^{-\frac{1}{2}}(\Sigma)\times H^{\frac{1}{2}}(\Sigma)$ such that
\begin{equation}\label{eg3}
\begin{array}{l}
a(\textbf{J},\textbf{j})+b(\textbf{j},M)=\ell(\textbf{j})\\\\
b(\textbf{J},m)+c(M,m)=0
\end{array}
\end{equation}
for all $(\textbf{j},m)\in \textbf{H}^{-\frac{1}{2}}(\Sigma)\times
H^{\frac{1}{2}}(\Sigma)$.\\

We now proceed to finite dimensional subspaces
$\mathcal{R}_{h}\subset \textbf{H}^{-\frac{1}{2}}(\Sigma)$ of dimension $n$
and $\mathcal{M}_{h}\subset H^{\frac{1}{2}}(\Sigma)$ of dimension
$m$, and seek approximations $\textbf{J}_{h}\in\mathcal{R}_{h}$ and
$M_{h}\in\mathcal{M}_{h}$ for $\textbf{J}$ and $M$, such that
\begin{equation}\label{eg4}
\begin{array}{l}
a(\textbf{J}_{h},\textbf{j})+b(\textbf{j},M_{h})=\ell(\textbf{j}),\\\\
b(\textbf{J}_{h},m)+c(M_{h},m)=0
\end{array}
\end{equation}
for all $\textbf{j}\in\mathcal{R}_{h}$ and $m\in\mathcal{M}_{h}$.\\
Let $\{\bpsi_{i}\}_{i=1}^{n}$ be a basis of $\mathcal{R}_{h}$ and
$\{\varphi_{j}\}_{j=1}^{m}$ be a basis of $\mathcal{M}_{h}$.
$\textbf{J}_{h}$ and $M_{h}$ are of the forms
\begin{equation}\label{eg5}
\textbf{J}_{h}:=\sum_{i=1}^{n}\lambda_{i}\bpsi_{i}\hspace{0.2cm}\mbox{and}\hspace{0.2cm}M_{h}:=\sum_{j=1}^{m}\mu_{j}\varphi_{j}.
\end{equation}
Inserting (\ref{eg5}) in (\ref{eg4}) provides
\begin{equation}\label{eg6}
\begin{aligned}
\sum_{i=1}^{n}\lambda_{i}a(\bpsi_{i},\bpsi_{k})+\sum_{j=1}^{m}\mu_{j}b(\bpsi_{k},\varphi_{j})&=\ell(\bpsi_{k})\\\\
\sum_{i=1}^{n}\lambda_{i}b(\bpsi_{i},\varphi_{l})+\sum_{j=1}^{m}\mu_{j}c(\varphi_{j},\varphi_{l})&=0
\end{aligned}
\end{equation}
for all $\bpsi_{k}$ and $\varphi_{l}$, $1\leq k\leq n$, $1\leq l\leq
m$.\\
With matrices and vectors
\begin{equation}\label{eg7}
\begin{array}{l}
A:=(a(\bpsi_{i},\bpsi_{k}))_{i,k}\in\mathbb{C}^{n\times n},\\\\
B:=(b(\bpsi_{i},\varphi_{l}))_{i,l}\in\mathbb{C}^{n\times m},\\\\
C:=(c(\varphi_{j},\varphi_{l}))_{j,l}\in\mathbb{C}^{m\times m},\\\\
\blambda:=(\lambda_{i})_{i}\in\mathbb{C}^{n},\\\\
\bmu:=(\mu_{j})_{j}\in\mathbb{C}^{m},\\\\
\bell:=(\ell(\bpsi_{k}))_{k}\in\mathbb{C}^{n}.
\end{array}
\end{equation}
(\ref{eg6}) has also the form
\begin{equation}\label{eg8}
\left(\begin{array}{cc}
A&B^{t}\\
B&C
\end{array}\right)\left(\begin{array}{c}
\blambda\\
\bmu
\end{array}\right)=\left(\begin{array}{c}
\bell\\
0
\end{array}\right).
\end{equation}

We have considered with $\{\bpsi_{i}\}_{i=1}^{n}$ a basis of $\mathcal{R}_{h}$ and $\{\varphi_{j}\}_{j=1}^{m}$ a basis of $\mathcal{M}_{h}$. These functions, are chosen as piecewise polynomials. To win these bases, we consider suitable basis functions locally on the element of a grid, i.e. on each component grid.\\
If we start from a grid
$$\{\Sigma_{k}\}_{k=1}^{N}\hspace{0.3cm}\mbox{with}\hspace{0.3cm}\bigcup_{1\leq k\leq N}\Sigma_{k}=\Sigma$$
with $N$ elements, and let $\{\widehat{\bpsi}_{i}\}_{i=1}^{\widehat{n}}$ and $\{\widehat{\varphi}_{j}\}_{j=1}^{\widehat{m}}$ respectively bases on a square reference element $\widehat{\Sigma}$. The local basis functions on an element $\Sigma_{k}$ are each $\{\bpsi_{i}\}_{i=1}^{n_{k}}$ or $\{\varphi_{j}\}_{j=1}^{m_{k}}$.\\

It should therefore be calculated first
$$A:=(a(\bpsi_{j_{s}},\bpsi_{i_{z}}))_{i_{z},j_{s}}\in \mathbb{C}^{n\times n},$$
where $\bpsi_{j_{s}}$ or $\bpsi_{i_{z}}$ are the basics function of
$\mathcal{R}_{h}$ and
$$a(\bpsi_{j_{s}},\bpsi_{i_{z}})=\int_{\Sigma}\mathcal{V}_{\alpha}(\bpsi_{j_{s}})_{T}\cdot\bpsi_{i_{z}}\hspace{0.1cm}dS=\sum_{k=1}^{N}\int_{\Sigma_{k}}\mathcal{V}_{\alpha}(\bpsi_{j_{s}})_{T}\cdot\bpsi_{i_{z}}\hspace{0.1cm}dS,$$

We test each local basis function against any other local basis function and sum the result to the test value of the global basis functions, which include these local basis functions.\\
Let $I_{N}=\{1,\ldots,N\}$ the index set for the grid elements, $I_{\widehat{n}}=\{1,\ldots,\widehat{n}\}$ the index set for the basic functions on the reference element and $I_{n}=\{1,\ldots,n\}$ the index set for the global basis functions.\\
Let $\bzeta:I_{N}\times I_{\widehat{n}}\rightarrow I_{n}$ the mapping from local to global basis functions such that $\bzeta(k,i)=j$, if the local basis function $\bpsi_{k,i}$ component of the global basis function is $\bpsi_{j}$.\\
Let $\bzeta^{-1}$ the set of all pairs of $(k,j)$ with $\bzeta(k,j)=i$, then
$$\int_{\Sigma}\mathcal{V}_{\alpha}(\bpsi_{j_{s}})_{T}\cdot\bpsi_{i_{z}}\hspace{0.1cm}dS=\sum_{\substack{(k,i)\in\\\bzeta^{-1}(i_{z})}}\sum_{\substack{(l,j)\in\\\bzeta^{-1}(j_{s})}}\int_{\Sigma_{k}}\mathcal{V}_{\alpha}(\bpsi_{l,j})_{T}\cdot\bpsi_{k,i}\hspace{0.1cm}dS$$
$$=\sum_{\substack{(k,i)\in\\\bzeta^{-1}(i_{z})}}\sum_{\substack{(l,j)\in\\\bzeta^{-1}(j_{s})}}\int_{\Sigma_{k}}\int_{\Sigma_{l}}G_{\alpha}(\lvert\textbf{x}-\textbf{y}\rvert)(\bpsi_{l,j}(\textbf{y}))^{t}\cdot\bpsi_{k,i}(\textbf{x})\hspace{0.1cm}dS_{\textbf{y}}\hspace{0.1cm}dS_{\textbf{x}}.$$
We are dealing in this implementation with Raviart-Thomas basis
functions. The transformation of these functions requires a
Peano transformation
$\bpsi_{k,i}=\dfrac{1}{\lvert\mbox{det}\hspace{0.1cm}A_{k}\rvert}A_{k}\widehat{\bpsi}_{i}$. Thus, if $A_{k}=(\textbf{a}_{1},\textbf{a}_{2})$, $\mbox{det}A_{k}$ is calculated by $\mbox{det}A_{k}=(\textbf{a}_{1}\times\textbf{a}_{2})\cdot\dfrac{\textbf{a}_{1}\times\textbf{a}_{2}}{\lVert\textbf{a}_{1}\times\textbf{a}_{2}\rVert}$.
The Peano-transformation of the local basis functions to the basic
functions on the reference element then gives
\begin{equation}\label{eg9}
\begin{aligned}
I&=\sum_{\substack{(k,i)\in\\\bzeta^{-1}(i_{z})}}\sum_{\substack{(l,j)\in\\\bzeta^{-1}(j_{s})}}\int_{\Sigma_{k}}\int_{\Sigma_{l}}G_{\alpha}(\lvert\textbf{x}-\textbf{y}\rvert)(\bpsi_{l,j}(\textbf{y}))^{t}\cdot\bpsi_{k,i}(\textbf{x})\hspace{0.1cm}dS_{\textbf{y}}\hspace{0.1cm}dS_{\textbf{x}}\\\\
&=\sum_{\substack{(k,i)\in\\\bzeta^{-1}(i_{z})}}\sum_{\substack{(l,j)\in\\\bzeta^{-1}(j_{s})}}\int_{\widehat{\Sigma}}\int_{\widehat{\Sigma}}\dfrac{G_{\alpha}(\lvert\textbf{x}-\textbf{y}\rvert)}{\lvert\mbox{det}\hspace{0.1cm}A_{k}\cdot\mbox{det}\hspace{0.1cm}A_{l}\rvert}(\widehat{\bpsi}_{i}(\widehat{\textbf{x}}))^{t}(A_{k})^{t}\cdot
A_{l}\widehat{\bpsi}_{j}(\widehat{\textbf{y}})\hspace{0.1cm}dS_{\widehat{\textbf{y}}}\hspace{0.1cm}dS_{\widehat{\textbf{x}}}
\end{aligned}
\end{equation}
with $\textbf{x}=\textbf{a}_{k}+A_{k}\widehat{\textbf{x}}$ and $\textbf{y}=\textbf{a}_{l}+A_{l}\widehat{\textbf{y}}$, and referent element $\widehat{\Sigma}$.\\

The calculation of the integrals with Helmholtz kernel $G_{\alpha}$
is not exact. We consider the expansion of the Helmholtz kernel in a Taylor
series. There holds
\begin{equation*}
\begin{aligned}
\nonumber G_{\alpha}(\lvert\textbf{x}-\textbf{y}\rvert)  &=\dfrac{1}{4\pi}
\frac{e^{\alpha
i\lvert\textbf{x}-\textbf{y}\rvert}}{\lvert\textbf{x}-\textbf{y}\rvert}
&=\dfrac{1}{4\pi}\left[ \frac{1}{\lvert\textbf{x}-\textbf{y}\rvert}+\alpha i+\frac{(\alpha
i)^2}{2}\lvert\textbf{x}-\textbf{y}\rvert+\dots\right]
\end{aligned}
\end{equation*}
The first term  are singular for $\textbf{x}=\textbf{y}$ and the correspondly integral are treated by analytic evaluation in \textit{Maiprogs}, cf. Maischak \cite{Mattias1,Mattias2,Mattias3} , but the integrals of all other summands can be calculated sufficiently well by Gaussian quadrature.\\
We compute
\begin{equation}\label{eq0}
\begin{aligned}
b(\bpsi_{i_{z}},\varphi_{j_{s}})&= -\int_\Sigma
\mathcal{V}_{\alpha}(\nabla_{T}\cdot\bpsi_{i_{z}})\cdot\varphi_{j_{s}}
\hspace{0.1cm}dS\\\\
&=-\sum_{\substack{(k,i)\in\\\bzeta^{-1}_{\psi}(i_{z})}}\sum_{\substack{(l,j)\in\\\bzeta^{-1}_{\varphi}(j_{s})}}
\int_{\Sigma_l}\int_{\Sigma_k}
G_{\alpha}(\lvert\textbf{x}-\textbf{y}\rvert)\nabla_{T}\cdot\bpsi_{k,i}(\textbf{y})
\cdot\varphi_{l,j}(\textbf{x})
\hspace{0.1cm}dS_{\textbf{y}}\,\text{d}S_{\textbf{x}}.
\end{aligned}
\end{equation}
with $\bzeta^{-1}_{\psi}=\bzeta$ described above, and $\bzeta^{-1}_{\varphi}$, the analogously defined map for the basic functions of $\mathcal{M}_{h}$.\\
While a transformation of the scalar basis functions is not required, the transformation of the surface divergence of Raviart-Thomas elements is carried out by $\nabla_{T}\cdot\bpsi_{k,i}=\frac{1}{\lvert\mbox{det}A_{k}\rvert}\widehat{\nabla}\cdot\widehat{\bpsi}_{i}$ and we have
\begin{equation}\label{eg18a}
\begin{aligned}
b(\bpsi_{i_{z}},\varphi_{j_{s}})=-\sum_{\substack{(k,i)\in\\\bzeta^{-1}_{\psi}(i_{z})}}\sum_{\substack{(l,j)\in\\\bzeta^{-1}_{\varphi}(j_{s})}}
\int_{\widehat{\Sigma}}\int_{\widehat{\Sigma}}\dfrac{G_{\alpha}(\lvert\textbf{x}-\textbf{y}\rvert)}{\lvert\mbox{det}A_{k}\rvert}\widehat{\nabla}\cdot\widehat{\bpsi}_{k,i}(\widehat{\textbf{y}})
\cdot\widehat{\varphi}_{l,j}(\widehat{\textbf{x}})
\hspace{0.1cm}dS_{\widehat{\textbf{y}}}\,\text{d}S_{\widehat{\textbf{x}}}
\end{aligned}
\end{equation}
with $\textbf{y}=\textbf{a}_{k}+A_{k}\widehat{\textbf{y}}$ and $\textbf{x}=\textbf{a}_{l}+A_{l}\widehat{\textbf{x}}$.\\

The calculation of $c(\varphi_{i},\varphi_{j})$ is similar to the above-mentioned case.\\

The calculation of the right hand side  appears simple at first glance,
since there are no single layer potential terms. Howewere we must compute the right hand side with quadrature.\\

The quadrature of an integral over $\textbf{f}$ on the reference
element is determined by the quadrature points
$\widehat{\textbf{x}}_{x,y}$, and the associated weights
$w_{x,y}=w_{x}\cdot w_{y}$,  which are processed in $x$ and $y$
direction. We perform the two-dimensional quadrature as a combination of one-dimensional
quadratures in each $x$ and $y$ direction, and we use here the
weights from the already implemented one-dimensional quadrature
formula. With $\widetilde{n}_{x}$ quadrature points in
$x$-direction, and $\widetilde{n}_{y}$ quadrature points in
$y$-direction, then the quadrature formula reads:
\begin{equation}\label{eg24}
\mathcal{Q}_{\widehat{\Sigma}}(\textbf{f})=\sum_{i=1}^{\widetilde{n}_{x}}\sum_{j=1}^{\widetilde{n}_{y}}\textbf{f}(\widehat{\textbf{x}}_{i,j})\cdot
w_{i}w_{j}.
\end{equation}
The quadrature points on the square reference element and the corresponding
weights for Gaussian quadrature are implemented in \textit{Maiprogs} already. For triangular elements,
we use Duffy transformation.\\

We will now calculate the right hand side  in the Galerkin formulation, i.e.
the linear form $\ell$, applied to the bases functions $\bpsi_{i}$,
$i=1,\ldots,n$. The quadrature takes place on the reference
element. We decompose the global into local basis functions
and then use the Peano-transformation for the Raviart-Thomas functions. It is therefore
\begin{equation*}
\begin{aligned}
\ell(\bpsi_{i_{r}})&=-\int_\Sigma (\textbf{E}^0_{T}(\textbf{x}))^{t}
\cdot\bpsi_{i_{r}}(\textbf{x})\hspace{0.1cm}dS_{\textbf{x}}\\\\
&=-\sum_{\substack{(k,i)\in\\\zeta^{-1}(i_{r})}}
\int_{\widehat{\Sigma}} (\textbf{E}^0_{T}(\textbf{x}))^{t} \cdot
A_{k}\cdot\widehat{\bpsi}_{k,i}(\widehat{\textbf{x}})\hspace{0.1cm}dS_{\widehat{\textbf{x}}}
\end{aligned}
\end{equation*}
with $\textbf{x}=\textbf{a}_k+A_k\widehat{\textbf{x}}$. Applying
(\ref{eg24}) leads with
$\widetilde{n}_{x}=\widetilde{n}_{y}:=\widetilde{n}$ to
\begin{equation}\label{eg26}
\mathcal{Q}(\ell(\bpsi_{i}))=-\sum_{\substack{(k,i)\in\\\zeta^{-1}(i_{r})}}\sum_{i_{1}=1}^{\widetilde{n}}\sum_{i_{2}=1}^{\widetilde{n}}(\textbf{E}^0_{T}(\textbf{x}_{i_{1},i_{2}}))^{t}
\cdot
A_{k}\cdot\widehat{\bpsi}_{k,i}(\widehat{\textbf{x}}_{i_{1},i_{2}})\cdot
w_{i_{1}}w_{i_{2}}
\end{equation}
with $\textbf{x}_{i,j}=\textbf{a}_k+A_k\widehat{\textbf{x}}_{i,j}$.
As before, the task is carried out by looping through all grid
components, and the values are added to the entries for each of its
base function.\\
The electrical field can be calculated by
\begin{equation}\label{eg34}
\textbf{E}_{h}=\mathcal{V}_{\alpha}(\textbf{J}_{h})+\mbox{grad}\hspace{0.1cm}\mathcal{V}_{\alpha}(M_{h}).
\end{equation}
We have for the first term in (\ref{eg34}) with $(\ref{eg5})_{1}$
\begin{equation}\label{eg36}
\mathcal{V}_{\alpha}(\textbf{J}_{h})(\textbf{x})=\sum_{i=1}^{n}\lambda_{i}\int_{\Sigma}G_{\alpha}(\lvert\textbf{x}-\textbf{y}\rvert)\bpsi_{i}(\textbf{y})dS_{\textbf{y}}.
\end{equation}
Then using Peano-transformation we have
\begin{equation}\label{eg37}
\begin{aligned}
\mathcal{V}_{\alpha}(\bpsi_{i_{s}})(\textbf{x})&=\int_{\Sigma}G_{\alpha}(\lvert\textbf{x}-\textbf{y}\rvert)\bpsi_{i_{s}}(\textbf{y})dS_{\textbf{y}}\\\\
&=\sum_{\substack{(l,i)\in\\\bzeta^{-1}(i_{s})}}\int_{\widehat{\Sigma}}\dfrac{G_{\alpha}(\lvert\textbf{x}-\textbf{y}\rvert)}{\lvert\mbox{det}\hspace{0.1cm}A_{l}\rvert}
A_{l}\widehat{\bpsi}_{i}(\widehat{\textbf{y}})\hspace{0.1cm}dS_{\widehat{\textbf{y}}}.
\end{aligned}
\end{equation}
For the second term in (\ref{eg34}) we have with
\begin{equation}\label{eg42}
\begin{aligned}
\mbox{grad}\hspace{0.1cm}\mathcal{V}_{\alpha}(\varphi_{j_{z}})(\textbf{x})&=\sum_{\substack{(l,j)\in\\\bzeta^{-1}(j_{z})}}\int_{\widehat{\Sigma}}\mbox{grad}_{\textbf{x}}G_{\alpha}(\lvert\textbf{x}-\textbf{y}\rvert)\widehat{\varphi}_{j}(\widehat{\textbf{y}})\hspace{0.1cm}dS_{\widehat{\textbf{y}}}
\end{aligned}
\end{equation}
The calculation of $\textbf{H}_{T}^{\pm}$ is done as follows (compare Remark \ref{slem1} $(v)$)
\begin{equation}\label{eg43}
\textbf{H}_{T}^{\pm}=\left[ \textbf{n}\times\mbox{curl}\hspace{0.1cm}\mathcal{V}_{\alpha}(\textbf{J})\right]^{\pm}=\pm\dfrac{1}{2}\textbf{J}(\textbf{x})+\dfrac{1}{2}\textbf{n}(\textbf{x})\times\int_{\Sigma}\mbox{grad}_{\textbf{x}}G_{\alpha}(\lvert\textbf{x}-\textbf{y}\rvert)\times\textbf{J}(\textbf{y}) dS_{\textbf{y}}.
\end{equation}

\section{Numerical experiments}
\label{sec:s5}

\begin{example}
Here, we consider one example to test the implementation. As domain we take the cube $\Omega_{-}=[-2,2]^{3}$, and we now want to test the Galerkin method in (\ref{eg4}). We choose the wave number $\alpha=0.1$ (or $\alpha=0.5,1.5$), and the exact solution
\begin{equation}\label{eg29}
\textbf{J}=\dfrac{1}{8}\left( \begin{array}{c}
0\\
(1-x_{1})(1-x_{2})\cdot n_{3}\\
-(1-x_{1})(1-x_{2})\cdot n_{2}
\end{array}
\right)
\end{equation}
and
\begin{equation}\label{eg30}
M=\dfrac{1}{8\alpha^{2}}(x_{1}-1)\cdot n_{3}
\end{equation}
where $\textbf{n}=(n_{1},n_{2},n_{3})$ denotes the outer normal vector at a point on the surface $\Sigma=\cup_{k=1}^{6}\Sigma_{k}$. We can write each term of equation (\ref{s55}) as:
\begin{equation}\label{eg31}
\begin{aligned}
\mathcal{V}_{\alpha}(\textbf{J})_{T}(\textbf{x})&=\sum_{k=1}^{6}\int_{\Sigma_{k}}G_{\alpha}(\lvert\textbf{x}-\textbf{y}\rvert)(\textbf{J}_{k}(\textbf{y}))^{t}\hspace{0.1cm}dS_{\textbf{y}},
\end{aligned}
\end{equation}
and
\begin{equation}\label{eg32}
\begin{aligned}
\mbox{grad}_{T}\mathcal{V}_{\alpha}(M)_{T}(\textbf{x})&=\sum_{k=1}^{6}\mbox{grad}_{T}\int_{\Sigma_{k}}G_{\alpha}(\lvert\textbf{x}-\textbf{y}\rvert)M_{k}(\textbf{y})\hspace{0.1cm}dS_{\textbf{y}}.
\end{aligned}
\end{equation}
Then, from (\ref{s55}), (\ref{eg31}) and (\ref{eg32}) holds
\begin{equation}\label{eg33}
\textbf{E}_{T}=\sum_{k=1}^{6}\left( \int_{\Sigma_{k}}G_{\alpha}(\lvert\textbf{x}-\textbf{y}\rvert)(\textbf{J}_{k}(\textbf{y}))^{t}\hspace{0.1cm}dS_{\textbf{y}}+\mbox{grad}_{T}\int_{\Sigma_{k}}G_{\alpha}(\lvert\textbf{x}-\textbf{y}\rvert)M_{k}(\textbf{y})\hspace{0.1cm}dS_{\textbf{y}}\right) .
\end{equation}

We use different values of $\alpha$ for our investigation. In Table \ref{table1} we present the results of the errors in energy norm and $L2$-norm for $\alpha=0.1,0.5,1.5$ for the uniform $h$-version with polynomial degree $p=1$. In Figures \ref{figure1} and \ref{figure2} we compare the $h$-version with different $\alpha$. The exact norm is known by extrapolation for $\alpha=0.1$ is $\lvert C\rvert=8.580798$, for $\alpha=0.5$ is $\lvert C\rvert=1.6171534$, and for $\alpha=1.5$ is $\lvert C\rvert=1.8042380$. Here $C=Re\langle\textbf{E}_{T}^{0},\textbf{J}\rangle$ and $C_{h}=Re\langle\textbf{E}_{T}^{0},\textbf{J}_{h}\rangle$ (see \cite{Holm}). The exact $L2$-norm is known by extrapolation for $\alpha=0.1$ are $\lVert \textbf{J}\rVert_{L^2}=2.1066356$ and $\lVert M\rVert_{L^2}=81.9249906$, for $\alpha=0.5$ are $\lVert \textbf{J}\rVert_{L^2}=2.1977966$ and $\lVert M\rVert_{L^2}=3.9588037$ and for $\alpha=1.5$ are $\lVert \textbf{J}\rVert_{L^2}=2.3826646$ and $\lVert M\rVert_{L^2}=0.7763804$.\\

The convergence rate $\eta$ for $\alpha=0.1$ are for the energy norm $\eta_{C}=1.325363$, for $L^{2}$-norm $\eta_{\textbf{J}}=1.617988$ and $\eta_{M}=1.184964$. For $\alpha=0.5$ are for the energy norm $\eta_{C}=1.165255$, for $L^{2}$-norm $\eta_{\textbf{J}}=0.976440$ and $\eta_{M}=1.211619$ and for $\alpha=1.5$ are for the energy norm $\eta_{C}=1.552163$, for $L^{2}$-norm $\eta_{\textbf{J}}=0.174124$ and $\eta_{M}=0.295586$.\\

Let as compare our numerical convergence rates above for the boundary element methods obtained in the above example with the theoretical convergence rates predicted by Theorem~\ref{ts1}. Note that we have implemented the boundary integral equation system (\ref{s55}), (\ref{s56}) and note the strongly elliptic system (\ref{b5}), where convergence is garanteed due to Theorem \ref{ts1}. Nevertheless our experiments show convergence for the boundary element solution, but with suboptimal convergence rates. Theorem \ref{ts1} predicts (when Raviart-Thomas elements are used to approximate $\textbf{J}$ and piecewise linear elements to approximate $M$) a convergence rate of order $\eta=\frac{3}{2}$ in the energy norm for smooth solutions $\textbf{J}$ and $M$. Our computations depend on the parameter $\alpha$ which is a well-known effect with boundary integral equations where it may come to spurious eigenvalues diminishing the orders of the Galerkin approximations. Due to the cube $\Omega_{-}=[-2,2]^{3}$ the numerical solution might become singular near the edges and corners of $\Omega_{-}$; hence the Galerkin scheme converge suboptimally.

\end{example}
\vspace{1cm}
\begin{table}[bht]
\begin{center}
\begin{tabular}[h]{|c||l|c|c|c|c|c|c|}\hline
N & DOF  & $\lvert C\rvert$      & $\lvert C-C_{h}\rvert$  &  $\lVert\textbf{J}\rVert_{\textbf{L}^{2}}$  &  $\lVert M\rVert_{\textbf{L}^{2}}$  &  $\lVert\textbf{J}-\textbf{J}_{h}\rVert_{\textbf{L}^{2}}$  &  $\lVert M-M_{h}\rVert_{\textbf{L}^{2}}$ \\ \hline \hline
  &      &           &           & $\alpha=0.1$ &             &                      &                                                                           \\ \hline
1 & 144  & 8.502965  & 1.153119  &  2.085189  & 80.704374     & 0.299829             &                                  14.08929                                 \\ \hline
2 & 576  & 8.568451  & 0.460150  &  2.104369  & 81.690279     & 0.097681             &                                   6.196968                                \\ \hline
3 & 2304 & 8.578833  & 0.033717  &  2.106395  & 81.879637     & 0.031823             &                                   2.725645                                \\ \hline
4 & 9216 & 8.654072  & 0.073274  &  2.117002  & 83.123825     & 0.010367             &                                   1.198835                                \\ \hline
  &      &           &           & $\alpha=0.5 $  &           &                      &                                                                           \\ \hline
1 & 144  & 1.603519  & 0.209552  &  2.149511  & 3.8937090     & 0.458159             &                                   0.714952                                \\ \hline
2 & 576  & 1.614451  & 0.093436  &  2.185426  & 3.9467491     & 0.232851             &                                   0.308704                                \\ \hline
3 & 2304 & 1.616616  & 0.041661  &  2.194608  & 3.9565591     & 0.118342             &                                   0.133293                                \\ \hline
4 & 9216 & 1.617260  & 0.018576  &  2.198619  & 3.9592220     & 0.060145             &                                   0.057554                                \\ \hline
  &      &           &           & $\alpha=1.5$  &            &                      &                                                                           \\ \hline
1 & 144  & 1.774450  & 0.326497  &  2.350909  & 0.7243729     & 0.387707             &                                   0.279375                                \\ \hline
2 & 576  & 1.800799  & 0.111334  &  2.365011  & 0.7422644     & 0.343627             &                                   0.227618                                \\ \hline
3 & 2304 & 1.803838  & 0.037965  &  2.382843  & 0.7539064     & 0.304558             &                                   0.185450                                \\ \hline
4 & 9216 & 1.804284  & 0.012946  &  2.397906  & 0.7909461     & 0.269932             &                                   0.151093                                \\ \hline
\end{tabular}
\end{center}
\caption{\label{table1}Errors in $L^{2}$-norm and energy norm with respect to the degrees of freedom for $\alpha=0.1,0.5,1.5$.}
\end{table}
% \newpage
% \vspace{1cm}
\begin{figure}[htb]
\centerline{\resizebox{8cm}{!}{\includegraphics*{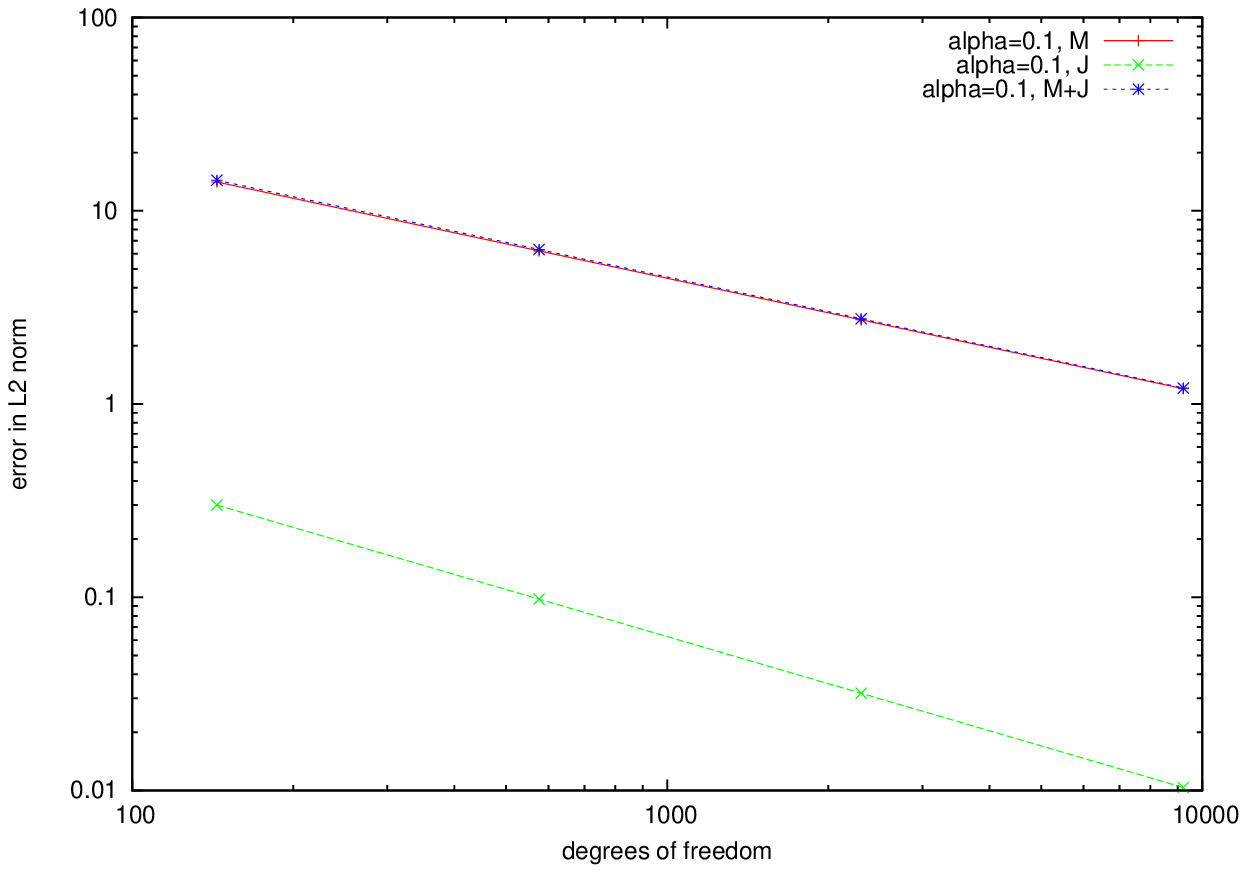}}\quad
\resizebox{8cm}{!}{\includegraphics*{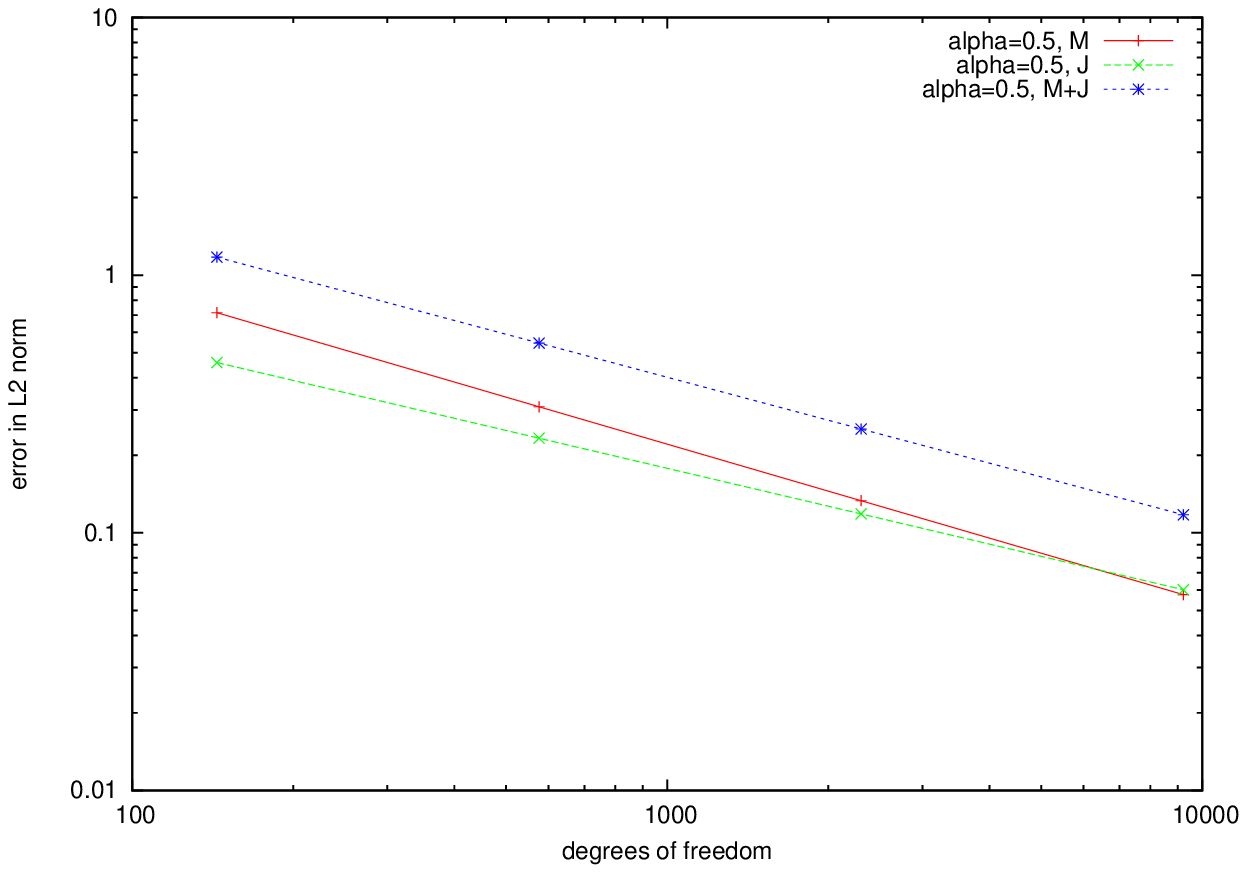}}}
\caption{\label{figure1}Errors in L2-norm for  $\alpha=0.1,0.5$}
\end{figure}
%\vspace{2cm}
\begin{figure}[htb]
\centerline{\resizebox{8cm}{!}{\includegraphics*{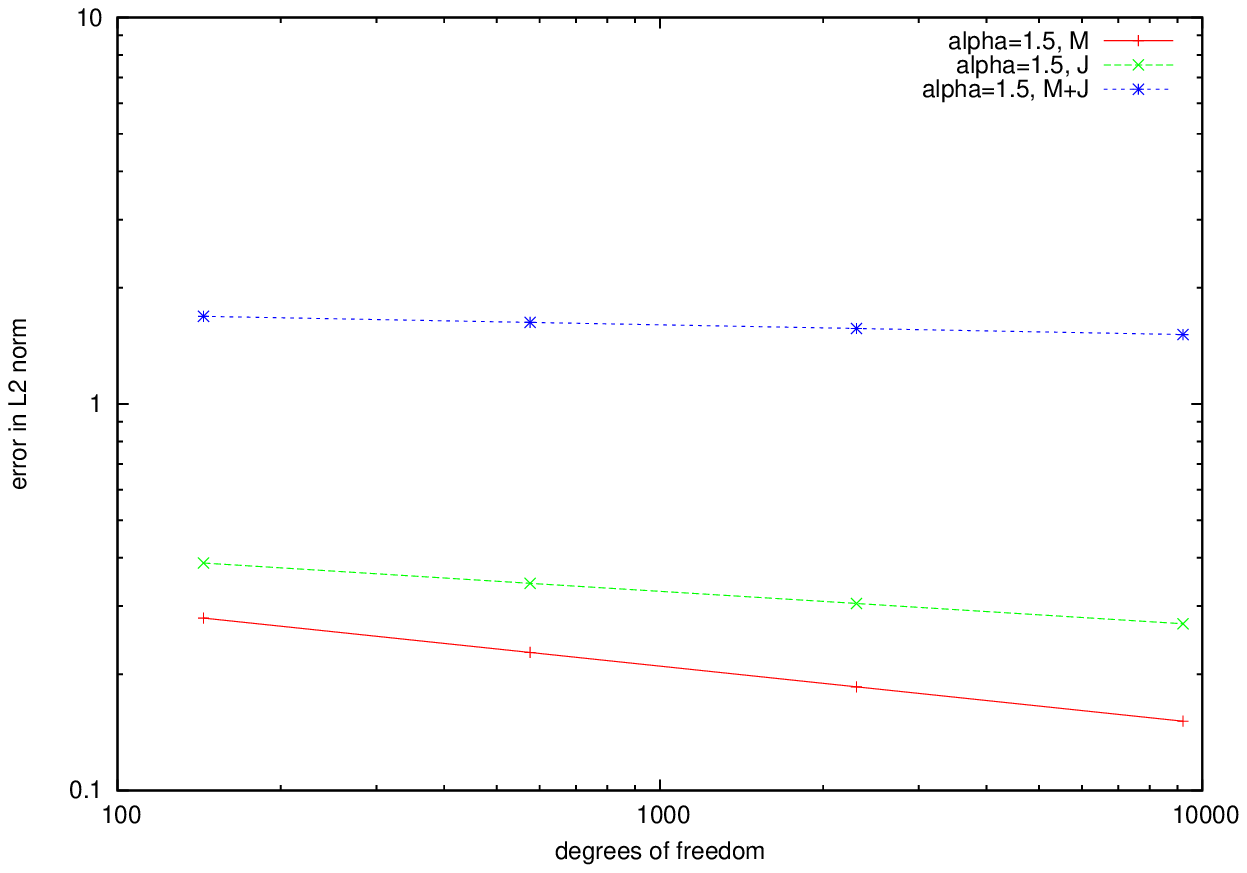}}\quad
\resizebox{8cm}{!}{\includegraphics*{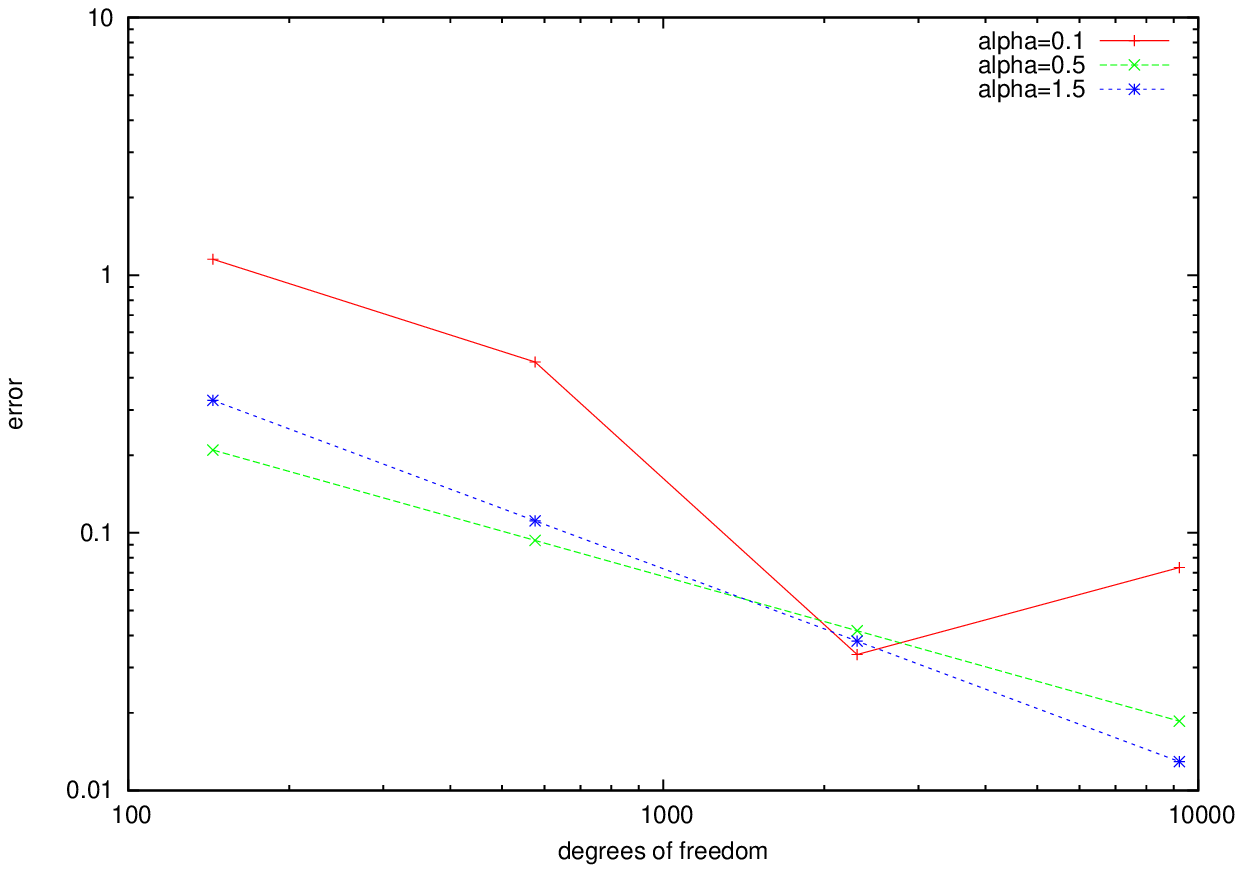}}}
\caption{\label{figure2}Errors in L2-norm for  $\alpha=1.5$ and energy norm $\lvert C-C_{h}\rvert=O(h^{\eta})$ for $\alpha=0.1,0.5,1.5$.}
\end{figure}
\newpage
Next, we apply the boundary element method above to compute the first terms in the asymptotic expansion of the electrical field considered in subsection \ref{sec:s1} (Remark \ref{rem1}). In this way we obtain good results for the electrical field at some point away from the transmission surface $\Sigma$ by only computing a few terms in the expansion.\\

Algorithm for the asymptotics of the eddy current problem:
\begin{enumerate}
 \item First solve the exterior Problem $(\textbf{P}_{\alpha \infty})$ by integral equations (\ref{s55}) and (\ref{s56}) i.e. (\ref{eg2}) with given incident field $-\textbf{E}_{T}^{0}$.
 \item Compute $\textbf{H}_{T}^{+}$ from (\ref{eg43}).
 \item Go back to 1: Solve the exterior problem $(\textbf{P}_{\alpha \infty})$ with new right hand side from (\ref{s31}).
 \item Go back to 2.
 \item $\textbf{E}=\textbf{E}_{0}+\beta^{-1}\textbf{E}_{1}+\beta^{-2}\textbf{E}_{2}+\textbf{R}_{m}$, where $\textbf{E}_{0}$ is the solution of the step 1 and $\textbf{E}_{1}$ and $\textbf{E}_{2}$ are solutions of step 3.
\end{enumerate}
We have $\widetilde{\textbf{E}}=\textbf{E}_{0}+\beta^{-1}\textbf{E}_{1}+\beta^{-2}\textbf{E}_{2}$ and calculate the error $\lvert\widetilde{\textbf{E}}-\textbf{E}_{\mbox{exact}}(\textbf{x}_{i})\rvert$, $i=1,2,3$, where $\textbf{x}_{1}=(3,0,0)$, $\textbf{x}_{2}=(6,0,0)$ and $\textbf{x}_{3}=(9,0,0)$. We present the results in Table \ref{table2} and in Figure \ref{figure3}.

\begin{table}[bht]
\begin{center}
\begin{tabular}[h]{|c||c||c||c|}\hline
DOF &  $\lvert\widetilde{\textbf{E}}-\textbf{E}_{\mbox{exact}}(\textbf{x}_{1})\rvert$ & $\lvert\widetilde{\textbf{E}}-\textbf{E}_{\mbox{exact}}(\textbf{x}_{2})\rvert$ & $\lvert\widetilde{\textbf{E}}-\textbf{E}_{\mbox{exact}}(\textbf{x}_{3})\rvert$ \\ \hline \hline
144  & 0.4959 & 0.6499 & 0.8049 \\ \hline
576  & 0.1043 & 0.0910 & 0.0347 \\ \hline
2304 & 0.0998 & 0.0067 & 0.0378 \\ \hline
\end{tabular}
\end{center}
\caption{\label{table2}Errors for electrical field in $\textbf{x}_{1}$, $\textbf{x}_{2}$, and $\textbf{x}_{3}$.}
\end{table}

\begin{figure}[ht]
\begin{center}
\includegraphics[width=8.5cm]{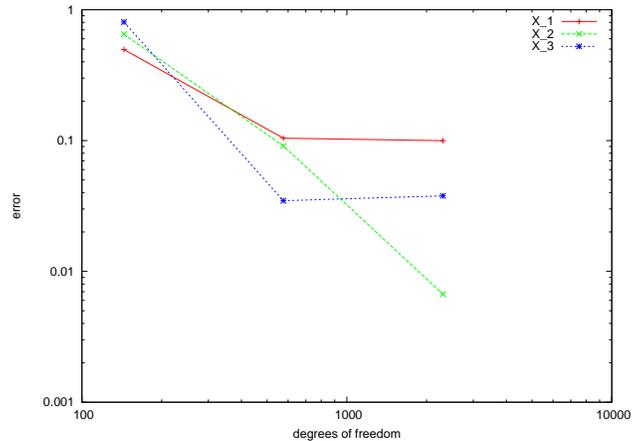}
\caption{\label{figure3}Errors for electrical field with respect to the degrees of freedom for $\textbf{x}_{1}$, $\textbf{x}_{2}$, and $\textbf{x}_{3}$.}
\end{center}
\end{figure}

\newpage

{\bf Acknowledgements:} This research was supported in part by the Progama ALECOL-DAAD, Institute for Applied Mathematics, Leibniz University of Hannover, Hannover-Germany, Department of Mathematics Sciences, Brunel University, U.K and Universidad del Norte, Barranquilla-Colombia. Also We thank the anonymous referees for their suggestions.\\

\end{document}